\documentclass{article}[12pt]
\usepackage{latexsym,amsmath}
\usepackage{amsfonts}
\usepackage{graphics}
\usepackage{amsmath,amssymb}
\usepackage[latin1]{inputenc}
\usepackage{epsfig}
\usepackage{graphicx}
\usepackage{amsthm}
\usepackage[linktocpage,colorlinks=true,linkcolor=red,citecolor=blue,bookmarks=true]{hyperref}

\newtheorem{theorem}{Theorem}[section]
\newtheorem{corollary}{Corollary}[section]

\newtheorem{proposition}{Proposition}[section]
\newtheorem{definition}{Definition}[section]

\newtheorem{remark}{Remark}[section]

\catcode`@=11 \@addtoreset{equation}{section} \catcode`@=12

\begin{document}
\thispagestyle{empty} \setcounter{page}{1}

\begin{center}
{\Large\bf On more general forms of proportional fractional operators}

\vskip.20in
  Fahd Jarad$^{a}$, Manar A. Alqudah$^{b}$,  Thabet Abdeljawad$^{c,d}$\\
{\footnotesize
$^{a}$Department of Mathematics, \c{C}ankaya University, 06790 Ankara, Turkey\\
 email: fahd@cankaya.edu.tr\\
$^{b}$ Department of Mathematical Sciences, Princess Nourah Bint Abdulrahman University\\ P.O. Box 84428, Riyadh 11671, Saudi Arabia.\\
email:maalqudah@pnu.edu.sa.\\
$^{c}$Department of Mathematics and Physical Sciences,
Prince Sultan
University\\ P. O. Box 66833,  11586 Riyadh, Saudi Arabia\\
email:tabdeljawad@psu.edu.sa\\
$^d$  Department of Medical Research, China Medical University, 40402, Taichung, Taiwan 
 
 }
\end{center}
\vskip.2in
{\footnotesize {\noindent \textbf{Abstract}
 In this article, we propose  new  proportional fractional operators generated from local proportional  derivatives of  a function with respect to another function. We present some properties of these fractional operators which can be also called proportional fractional operators of a function with respect to another function or proportional fractional operators with dependence on a kernel function.\\
\textbf{Keywords:} Proportional derivative of a function with respect to another function, general fractional proportional  integrals, general fractional proportional  derivatives.}}
\section{Introduction}

The fractional calculus, which is engaged in integral and differential operators of arbitrary  orders, is as old as the conceptional calculus that deals with integrals and derivatives of non-negative integer orders. Since not all of the real phenomena can be modeled using the operators in the traditional calculus, researchers searched for generalizations of these operators. It turned out that the fractional operators are excellent  tools to use in modeling long-memory processes and many phenomena that appear in physics, chemistry, electricity, mechanics and many other disciplines. Here, we invite the readers to read \cite{podlubny,Samko,f1,f222,f2,f3} and the reference cited in these books. However, for the sake of better understanding and modeling real world problems,  researchers were in need of other types of fractional operators that were confined to Riemann-Liouville fractional operators. In the literature, one can find many works that propose new fractional operators. We mention \cite{had,Kat1,Kat2,fahd3,fahd1,fahd11}. Nonetheless, the fractional integrals and derivatives which were proposed  in these works were just particular cases of what so called fractional integrals/derivatives of a function with respect to another function \cite{Samko,f2,fahd10}. There are  other types of fractional operators which were suggested in the literature. 

On the other hand, due to the singularities found in the traditional fractional operators which are thought to make some  difficulties in the modeling process, some researches recently proposed  new types of non-singular fractional operators. Some of these operators contain exponential kernels and some of them involve the Mittag-Leffler functions. For such types of fractional operators we refer to \cite{FCaputo,Losada,TD ROMP,Abdon,TD JNSA}.

All the fractional operators considered in the  references in the first and the second paragraphs are non-local. However, there are many local operators  found in the literature that allow differentiation to a non-integer order and these are called local fractional operators.  In \cite{kh}, Khalil et al.  introduced the so called conformable (fractional) derivative.  The author in \cite{T11} presented other basic concepts of conformable  derivatives. We would like to mention that the fractional operators proposed in \cite{Kat1,Kat2} are the non-local fractional version of the local operators suggested in \cite{kh}. In addition, the  non-local fractional version of the ones in \cite{T11} can be seen in \cite{fahd11}.

It is customary  that  any derivative  of order 0 when performed to a function should give  the function itself. This essential property is dispossessed by the conformable derivatives. Notwithstanding, in \cite{Anderson1,Anderson2}, the authors   introduced a newly defined local derivative that tend to the original function as the order  tends to zero and hence improved the conformable derivatives. In addition to this, the non-local fractional operators that emerge  from iterating the  above-mentioned  derivative were held forth in \cite{fahd12}.

Motivated by the above mentioned background, we extend the work done in \cite{fahd12} introduce a new generalized fractional calculus based on  the proportional derivatives of a function with respect to another function in paralel  with the definition discussed in \cite{Anderson1}.  The kernel obtained in the fractional operators which will be proposed contains an exponential function and is function dependent.  The  semi--group properties will be discussed.

The article is organized as follows: Section 2 presents  some essential definitions for fractional derivatives and integrals.  In Section 3 we present the general forms of the fractional proportional integrals and derivatives. In section 4, we present the general form of Caputo fractional proportional derivatives. In the end, we conclude our results.

\section{Preliminaries}
In this section, we present some essential definitions of some fractional derivatives and integrals. We first present the traditional fractional operators and then the fractional proportional operators.

\subsection{The conventional fractional operators and their general forms}
For $\alpha \in \mathbb{C},~Re(\alpha)>0$, 
 the left Riemann--Liouville fractional integral of order $\alpha $  has the f form
\begin{equation}\label{001}
(_{a}I^\alpha f)(x)=\frac{1}{\Gamma(\alpha)}\int_a^x (x-u)^{\alpha-1}f(u)du.
\end{equation}
  
  The  right Riemann--Liouville fractional integral of order $\alpha >0$ is  defined by
\begin{equation}\label{002}
(I_b^\alpha f)(x)=\frac{1}{\Gamma(\alpha)}\int_x^b (u-x)^{\alpha-1}f(u)du.
\end{equation}
 The left Riemann--Liouville fractional derivative of order $\alpha, Re(\alpha)\geq 0 $ is given as
 \begin{equation}\label{003}
 (_{a}D^\alpha f)(x)=\Big(\frac{d}{dx}\Big)^n(_{a}I^{n-\alpha} f)(x),~~n=[\alpha]+1.
 \end{equation}
  The right Riemann--Liouville fractional derivative of order $\alpha, Re(\alpha)\geq 0 $ reads
\begin{equation}\label{004}
(D_b^\alpha f)(t)=\Big(-\frac{d}{dt}\Big)^n(I_b^{n-\alpha} f)(t).
\end{equation}

The left Caputo fractional derivative  has the following form
\begin{equation}\label{005}
  (_{a}^{C}D^\alpha f)(x)=\big(_{a}I^{n-\alpha} f^{(n)}\big)(x),~~n=[\alpha]+1.
\end{equation}

The  right Caputo fractional derivative   becomes
\begin{equation}\label{006}
(^CD_b^\alpha f)(x)=\big(I_b^{n-\alpha}(-1)^nf^{(n)}\big)(x).
\end{equation}

The generalized left and right fractional integrals in the sense of Katugampola \cite{Kat1} are given respectively as
\begin{equation}\label{015}
(_{a}\textbf{I}^{\alpha,\rho} f)(x)=\frac{1}{\Gamma(\alpha)}\int_a^x(\frac{x^\rho-u^\rho}{\rho})^{\alpha-1} f(u)\frac{du}{u^{1-\rho}}
\end{equation}
and
\begin{equation}\label{016}
(\textbf{I}_{b}^{\alpha,\rho}f)(x)=\frac{1}{\Gamma(\alpha)}\int_t^b (\frac{u^\rho- x^\rho}{\rho})^{\alpha-1} f(u)\frac{du}{u^{1-\rho}}.
\end{equation}
The  generalized left and right fractional  derivatives in the sense of Katugampola \cite{Kat2}  are defined respectively as
\begin{eqnarray}\label{017}\nonumber
(_{a}\textbf{D}^{\alpha,\rho} f)(x)&=&\gamma^n(_{a}\textbf{I}^{n-\alpha,\rho} f)(t)\\&=&\frac{\gamma^n}{\Gamma(n-\alpha)}\int_a^x(\frac{x^\rho-u^\rho}{\rho})^{n-\alpha-1} f(u)\frac{du}{u^{1-\rho}}
\end{eqnarray}
and
\begin{eqnarray}\label{018}\nonumber
(\textbf{D}_{b}^{\alpha,\rho} f)(x)&=& (-\gamma)^n(\textbf{I}_b^{n-\alpha,\rho} f)(x)\\
&=&\frac{(-\gamma)^n}{\Gamma(n-\alpha)}\int_x^b(\frac{u^\rho-x^\rho}{\rho})^{n-\alpha-1} f(u)\frac{du}{u^{1-\rho}},
 \end{eqnarray}
where $\rho>0$ and   $\gamma=x^{1-\rho}\frac{d}{dx}$.
The Caputo modification of the left and right generalized fractional derivatives in the sense of Jarad et al. \cite{fahd3} are presented respectively as
\begin{eqnarray}\label{019}\nonumber
 (_{a}^C\textbf{D}^{\alpha,\rho} f)(x)&=&(_{a}\textbf{I}^{n-\alpha,\rho}\gamma^n f)(x)\\&=&\frac{1}{\Gamma(n-\alpha)}\int_a^x(\frac{x^\rho-u^\rho}{\rho})^{n-\alpha-1}\gamma^n f(u)\frac{du}{u^{1-\rho}},
\end{eqnarray}
and
\begin{eqnarray}\label{020}\nonumber
(^C\textbf{D}_{b}^{\alpha,\rho} f)(x)&=& (_{a}\textbf{I}^{n-\alpha,\rho}(-\gamma)^n f)(x)\\
&=&\frac{1}{\Gamma(n-\alpha)}\int_x^b(\frac{u^\rho-x^\rho}{\rho})^{n-\alpha-1} (-\gamma)^nf(u)\frac{du}{u^{1-\rho}}.
\end{eqnarray}
For $\alpha \in \mathbb{C},~Re(\alpha)>0$ the left Riemann-Liouville fractional integral of order $\alpha $ of $f$ with respect to a continuously differentiable  and increasing function $g$  has the following form \cite{Samko,f2}  
\begin{equation}\label{3}  
~_{a}I^{\alpha,g} f(x)=\frac{1}{\Gamma(\alpha)}\int_{a}^x \Big(g(x)-g(u)\Big)^{\alpha-1}f(u)g'(u)du.
\end{equation}
For $\alpha \in \mathbb{C},~Re(\alpha)>0$ the right Riemann-Liouville fractional integral of order $\alpha $ of $f$ with respect to a continuously differentiable  and increasing function $g$  has the following form \cite{Samko,f2} 
\begin{equation}\label{333}  
I_b^{\alpha,g} f(x)=\frac{1}{\Gamma(\alpha)}\int_{x}^b \Big(g(u)-g(x)\Big)^{\alpha-1}f(u)g'(u)du.
\end{equation}
 For $\alpha \in \mathbb{C},~Re(\alpha)\geq 0$,  the generalized left and right Riemann-Liouville fractional derivative of order $\alpha $ of $f$ with respect to a continuously differentiable  and increasing function $g$  have respectively the form \cite{Samko,f3}
 \begin{eqnarray}\label{4}\nonumber
 ~_{a}D^{\alpha,g}  f(x)&=&\Big(\frac{1}{g'(x)}\frac{d}{dx}\Big)^n(~_{a}I^{n-\alpha,g} f)(x)\\&=&\frac{\Big(\frac{1}{g'(x)}\frac{d}{dx}\Big)^n}{\Gamma(n-\alpha)}\int_{a}^x \Big(g(x)-g(u)\Big)^{n-\alpha-1}f(u)g'(u)du
 \end{eqnarray}
 \noindent and
\begin{eqnarray}\label{444}\nonumber
 D_b^{\alpha,g}  f(x)&=&\Big(-\frac{1}{g'(x)}\frac{d}{dx}\Big)^n(I_b^{n-\alpha,g} f)(x)\\&=&\frac{\Big(-\frac{1}{g'(x)}\frac{d}{dx}\Big)^n}{\Gamma(n-\alpha)}\int_{a}^x \Big(g(x)-g(u)\Big)^{n-\alpha-1}f(u)g'(u)du,
 \end{eqnarray}
 \noindent where $n=[\alpha]+1$. It is easy to observe that if we choose $g(x)=x$, the integrals 
 in (\ref{3}) and (\ref{333}) becomes the left and right Riemann-Liouville fractional integrals respectively and (\ref{4}) and (\ref{444}) becomes the left and right Riemann-Liouville fractional derivatives. When $g(x)=\ln x$, the Hadamard fractional operators are obtained \cite{Samko,f2}. While if one considers $g(x)=\frac{x^\rho}{\rho}$, the fractional operators in the settings of Katugampola \cite{Kat1,Kat2} are derived. 
 
 In left and right generalized Caputo derivatives of a function with respect to another function are presented respectively as \cite{fahd10}
\begin{equation} \label{555}
 ~_{a}^CD^{\alpha,g}  f(x)=\Big(~_aI^{n-\alpha,g} f^{[n]}\Big)(x)
\end{equation}
\noindent  and 
\begin{equation} \label{666}
 ^CD_b^{\alpha,g}  f(x)=\Big(~_aI^{n-\alpha,g} (-1)^nf^{[n]}\Big)(x),
\end{equation}
\noindent where $\displaystyle  f^{[n]}(x)=\Big(\frac{1}{g'(x)}\frac{d}{dx}\Big)^nf(x)$.

\subsection{The proportional derivatives and their  fractional integrals and derivatives} 
The conformable derivative was first introduced by Khalil et al. in \cite{kh} and then explored by the
current author in \cite{T11}. In his distinctive paper \cite{Anderson1}, Anderson et al. modified the  conformable derivative by using the proportional derivative. Indeed, he gave the following definition.

\begin{definition} \label{D1} \emph{(Modified conformable derivatives)}
For $\rho \in [0,1]$, let the functions $\kappa_0, \kappa_1:[0,1]\times \mathbb{R}\rightarrow [0,\infty)$ be continuous such that for all $t \in \mathbb{R}$ we have $$\lim_{\rho\rightarrow 0^+}\kappa_1(\rho,t)=1,~\lim_{\rho\rightarrow 0^+}\kappa_0(\rho,t)=0, \lim_{\rho\rightarrow 1^-}\kappa_1(\rho,t)=0,~\lim_{\rho\rightarrow 1^-}\kappa_0(\rho,t)=1,$$
and $\kappa_1(\rho,t)\neq 0,~~\rho \in [0,1),~~\kappa_0(\rho,t)\neq 0,~~\rho \in (0,1]$. Then, the modified conformable differential operator of order $\rho$ is defined by
\begin{equation}\label{anndy}
D^\rho f(t)=\kappa_1(\rho,t) f(t)+\kappa_0(\rho,t) f^\prime(t).
\end{equation}
\end{definition}
The derivative given in (\ref{anndy}) is called a proportional derivative. For more details about the control theory of the proportional derivatives and its component functions $\kappa_0$ and $\kappa_1$, we refer the reader to \cite{Anderson1,Anderson2}.

Of special interest, we shall restrict ourselves to the case when $\kappa_1(\rho,t)=1-\rho$ and $\kappa_0(\rho,t)=\rho$.
Therefore, (\ref{anndy}) becomes
\begin{equation}\label{prop derivative}
  D^\rho f(t)=(1-\rho) f(t)+\rho f^\prime(t).
\end{equation}
Notice that $\lim_{\rho \rightarrow 0^+}D^\rho f(t)= f(t)$ and $\lim_{\rho \rightarrow 1^-}D^\rho f(t)= f^\prime(t)$. It is clear that the derivative (\ref{prop derivative})  is somehow more  general than  the conformable derivative which does not tend to the original function as $\rho$ tends to $0$. The associated  fractional proportional integrals are defined as \cite{fahd12}

\begin{definition}\label{left and right integrals}For $\rho>0$ and $\alpha \in \mathbb{C},~~Re(\alpha)>0$, the  left fractional proportional  integral of $f$   reads \begin{equation}\label{LPI}
        (_{a}I^{\alpha,\rho} f)(x)=  \frac{1}{\rho^\alpha \Gamma(\alpha)}\int_a^x e^{\frac{\rho-1}{\rho}(x-\tau)} (x-\tau)^{\alpha-1} f(\tau)d\tau
         \end{equation}
and the right one reads
        \begin{equation}\label{RPI}
          (I_b^{\alpha,\rho} f)(x)=  \frac{1}{\rho^\alpha \Gamma(\alpha)}\int_x^b  e^{\frac{\rho-1}{\rho}(\tau-x)} (\tau-x)^{\alpha-1} f(\tau)d\tau.
           \end{equation}

\end{definition}

\begin{definition}\label{Prop fractional derivatives}\cite{fahd12}
For $\rho>0$ and $\alpha \in \mathbb{C},~~Re(\alpha)\geq 0$, the left fractional proportional derivative is defined as 

\begin{eqnarray}\label{LPD}\nonumber
  (_{a}D^{\alpha,\rho}f)(x)&=& D^{n,\rho} ~_{a}I^{n-\alpha,\rho} f(x)\\&=&\frac{D_x^{n,\rho}}{\rho^{n-\alpha}\Gamma(n-\alpha)}
  \int_a^x e^{\frac{\rho-1}{\rho}(x-\tau)}(x-\tau)^{n-\alpha-1} f(\tau)d \tau.
\end{eqnarray}

 The right proportional fractional  derivative  is defined by \cite{fahd12}
 \begin{eqnarray}\label{RPD}\nonumber
   (D_b^{\alpha,\rho}f)(x)&=& ~_{\ominus}D^{n,\rho} I_b^{n-\alpha,\rho} f(x)\\&=&\frac{~_{\ominus}D^{n,\rho}}{\rho^{n-\alpha}\Gamma(n-\alpha)}
   \int_x^b e^{\frac{\rho-1}{\rho}(\tau-x)}(\tau-x)^{n-\alpha-1} f(\tau)d \tau,
 \end{eqnarray}

\noindent   where $n=[Re(\alpha)]+1$ and $\displaystyle (_{\ominus}D^\rho f)(t)=(1-\rho)f(t)-\rho f^\prime(t)$.

\end{definition}

Lastly, the left and right  fractional proportional derivatives in the Caputo settings respectively read \cite{fahd12}
\begin{eqnarray}\label{LPDC}\nonumber
  (_{a}^CD^{\alpha,\rho}f)(x)&=&  \Big(~_{a}I^{n-\alpha,\rho}D^{n,\rho}f\Big)(x)\\\nonumber &=&\frac{1}{\rho^{n-\alpha}\Gamma(n-\alpha)}
  \int_a^x e^{\frac{\rho-1}{\rho}(x-\tau)}(x-\tau)^{n-\alpha-1} (D^{n,\rho}f)(\tau)d \tau
  \\
\end{eqnarray}
\noindent 
\noindent and 
 \begin{eqnarray}\label{RPDC}\nonumber
   (^CD_b^{\alpha,\rho}f)(x)&=& \Big( I_b^{n-\alpha,\rho}{~_\ominus}D^{n,\rho}f\Big)(x)\\\nonumber &=&\frac{1}{\rho^{n-\alpha}\Gamma(n-\alpha)}
   \int_x^b e^{\frac{\rho-1}{\rho}(\tau-x)}(\tau-x)^{n-\alpha-1} (~_{\ominus}D^{n,\rho}f)(\tau)d\tau.\\
 \end{eqnarray}

\section{The fractional proportional derivative of a function with respect to another function}
\begin{definition} \label{D2} \emph{(The proportional derivative of a function with respect to anothor function)}\\
For $\rho \in [0,1]$, let the functions $\kappa_0, \kappa_1:[0,1]\times \mathbb{R}\rightarrow [0,\infty)$ be continuous such that for all $t \in \mathbb{R}$ we have $$\lim_{\rho\rightarrow 0^+}\kappa_1(\rho,t)=1,~\lim_{\rho\rightarrow 0^+}\kappa_0(\rho,t)=0, \lim_{\rho\rightarrow 1^-}\kappa_1(\rho,t)=0,~\lim_{\rho\rightarrow 1^-}\kappa_0(\rho,t)=1,$$
and $\kappa_1(\rho,t)\neq 0,~~\rho \in [0,1),~~\kappa_0(\rho,t)\neq 0,~~\rho \in (0,1]$. Let also $g(t)$ be a strictly  increasing continuous function.  Then, the proportional  differential operator of order $\rho$ of $f$ with respect to $g$ is defined by
\begin{equation}\label{eq1}
D^{\rho,g} f(t)=\kappa_1(\rho,t) f(t)+\kappa_0(\rho,t)\frac{f^\prime(t)}{g'(t)}.
\end{equation}
\end{definition}
 we shall restrict ourselves to the case when $\kappa_1(\rho,t)=1-\rho$ and $\kappa_0(\rho,t)=\rho$.
Therefore, (\ref{eq1}) becomes
\begin{equation}\label{eq2}
  D^{\rho,g} f(t)=(1-\rho) f(t)+\rho \frac{f^\prime(t)}{g'(t)}.
\end{equation}
The corresponding integral of \eqref{eq2}

\begin{equation}\label{eq3}
_{a}I^{1,\rho,g}f(t)=\frac{1}{\rho}\int_a^t e^{\frac{\rho-1}{\rho}(g(t)-g(s))}f(s)g'(s)ds,
\end{equation}
where we accept that $~_{a}I^{0,\rho}f(t)=f(t)$.

To produce a generalized type fractional integral depending on the proportional
  derivative, we proceed by induction through changing the order of integrals to show that
 \begin{eqnarray} \label{eq4}
 \nonumber
   (_{a}I^{n,\rho,g} f)(t) &=& \frac{1}{\rho}\int_a^t e^{\frac{\rho-1}{\rho}(g(t)-g(\tau_1))} g'(\tau_1)d \tau_1 \frac{1}{\rho} \int_a^{\tau_1}
   e^{\frac{\rho-1}{\rho}(g(\tau_1)-g(\tau_2))}g'(\tau_2)d \tau_2\cdot\cdot\cdot\\\nonumber&\cdot\cdot\cdot&\frac{1}{\rho} \int_a^{\tau_{n-1}} e^{\frac{\rho-1}{\rho}(g(\tau_{n-1})-g(\tau_n))}f(\tau_n)g'(\tau_n) d\tau_n\\
    &=&  \frac{1}{\rho^n \Gamma(n)}\int_a^t  e^{\frac{\rho-1}{\rho}(g(t)-g(\tau))} (g(t)-g(\tau))^{n-1} f(\tau)g'(\tau)d\tau.
 \end{eqnarray}
 Based on (\ref{eq4}), we can present the following general proportional  fractional integral.

\begin{definition}\label{general left and right integrals}For $\rho \in (0,1]$, $\alpha \in \mathbb{C},~~Re(\alpha)>0$, we define the left fractional  integral of $f$ with respect to $g$  by 
\begin{equation} \label{eq5}
        (_{a}I^{\alpha,\rho,g} f)(t)=  \frac{1}{\rho^\alpha \Gamma(\alpha)}\int_a^t  e^{\frac{\rho-1}{\rho}(g(t)-g(\tau))} (g(t)-g(\tau))^{\alpha-1} f(\tau)g'(\tau)d\tau.
         \end{equation}
The right fractional proportional integral ending at $b$ can be defined by
 \begin{equation}\label{eq6}
          (I_b^{\alpha,\rho,g} f)(t)=  \frac{1}{\rho^\alpha \Gamma(\alpha)}\int_t^b  e^{\frac{\rho-1}{\rho}(g(\tau)-g(t))} (g(\tau)-g(t))^{\alpha-1} f(\tau)g'(\tau)d\tau.
\end{equation}
\end{definition}

\begin{remark} To deal with the right proportional  fractional case we shall use the notation $$(_{\ominus}D^{\rho,g} f)(t):=(1-\rho)f(t)-\rho \frac{f^\prime(t)}{g'(t)}.$$
We shall also write $$(_{\ominus}D^{n,\rho,g} f)(t)= (\underbrace{_{\ominus}D^{\rho,g}~ _{\ominus}D^{\rho,g}\ldots~_{\ominus}D^{\rho,g}}_{\texttt{n times}} f)(t).$$
\end{remark}

\begin{definition}\label{general left and right derivatives}
For $\rho>0$, $\alpha \in \mathbb{C},~~Re(\alpha)\geq 0$ and $g\in C[a,b]$, where $g'(t)>0$,  we define the  general left fractional derivative of $f$ with respect to $g$ as

\begin{eqnarray}\label{eq7}\nonumber
  (_{a}D^{\alpha,\rho,g}f)(t)&=& D^{n,\rho,g} ~_{a}I^{n-\alpha,\rho,g} f(t)\\\nonumber
  &=&\frac{D_t^{n,\rho,g}}{\rho^{n-\alpha}\Gamma(n-\alpha)}
  \int_a^t e^{\frac{\rho-1}{\rho}(g(t)-g(\tau))}(g(t)-g(\tau))^{n-\alpha-1} f(\tau)g'(\tau)d \tau\\   
\end{eqnarray}

 and the  general right fractional derivative of $f$ with respect to $g$ as
 \begin{eqnarray}\label{eq8}\nonumber
   (D_b^{\alpha,\rho,g}f)(t)&=& ~_{\ominus}D^{n,\rho,g} I_b^{n-\alpha,\rho,g} f(t)\\\nonumber
   &=&\frac{~_{\ominus}D_t^{n,\rho,g}}{\rho^{n-\alpha}\Gamma(n-\alpha)}
   \int_t^b e^{\frac{\rho-1}{\rho}(g(\tau)-g(t))}(g(\tau)-g(t))^{n-\alpha-1} f(\tau)g'(\tau)d \tau,\\
 \end{eqnarray}

\noindent   where $n=[Re(\alpha)]+1$.

\end{definition}

   \begin{remark}\label{reduction}
   
   Clearly, if we let $\rho=1$ in Definition \ref{general left and right integrals} and Definition \ref{general left and right derivatives}, we obtain the 
\begin{itemize}
\item the Riemann-Liouville fractional operators  \eqref{001}, \eqref{002},\eqref{003} and \eqref{004} if $g(t)=t$.    
\item the fractional operators in the Katugampola setting\eqref{015}, \eqref{016}, \eqref{017} and \eqref{018}  if $\displaystyle g(t)=\frac{t^{\mu}}{\mu}$.
\item The Hadamard fractional operators if $g(t)=\ln t$ \cite{Samko,f2}.
\item The fractional  operators mentioned in \cite{fahd11} if $\displaystyle g(t)=\frac{(t-a)^{\mu}}{\mu}$.
\end{itemize}
\end{remark}  

\begin{proposition} \label{2.4}
Let $\alpha, \beta \in \mathbb{C}$ be such that $Re(\alpha)\geq 0$ and $Re(\beta)>0$. Then, for any $\rho>0$ we have
\begin{itemize}
  \item (a) ~$\big(_{a}I^{\alpha,\rho,g} e^{\frac{\rho-1}{\rho}g(x)} (g(x)-g(a))^{\beta-1}\big)(t)=\frac{\Gamma(\beta)}{\Gamma(\beta+\alpha)\rho^\alpha}e^{\frac{\rho-1}{\rho}g(t)}(g(t)-g(a))^{\alpha+\beta-1},$\\$~~~Re(\alpha)>0.$
  \item (b)~ $\big(I_b^{\alpha,\rho,g} e^{-\frac{\rho-1}{\rho}g(x)} (g(b)-g(x))^{\beta-1}\big)(t)=\frac{\Gamma(\beta)}{\Gamma(\beta+\alpha)\rho^\alpha}e^{-\frac{\rho-1}{\rho}g(t)}(g(b)-g(t))^{\alpha+\beta-1},$\\$~~~Re(\alpha)>0.$
  \item (c)~ $\big(_{a}D^{\alpha,\rho} e^{\frac{\rho-1}{\rho}g(x)} (g(x)-g(a))^{\beta-1}\big)(t)=\frac{\rho^\alpha\Gamma(\beta)}{\Gamma(\beta-\alpha)}e^{\frac{\rho-1}{\rho}g(t)}(g(t)-g(a))^{\beta-1-\alpha},$\\$~~~Re(\alpha)\geq 0.$
    \item (d)~$\big(D_b^{\alpha,\rho,g} e^{-\frac{\rho-1}{\rho}g(x)} (g(b)-g(x))^{\beta-1}\big)(t)=\frac{\rho^\alpha\Gamma(\beta)}{\Gamma(\beta-\alpha)}e^{-\frac{\rho-1}{\rho}g(t)}(g(b)-G
    g(t))^{\beta-1-\alpha},$\\$~~~Re(\alpha)\geq 0.$
\end{itemize}
\end{proposition}
\begin{proof}
The proofs of relations (a) and (b) are very easy to handle. We will prove  (c) while the proof of  (d) is analogous.

By the definition of the left proportional fractional  derivative and relation (a), we have
\begin{align*}
&\Big(_{a}D^{\alpha,\rho,g} e^{\frac{\rho-1}{\rho}g(x)} (g(x)-g(a))^{\beta-1}\Big)(t)\\ 
&=D^{n,\rho,g} \Big(_{a}I^{n-\alpha,\rho,g} e^{\frac{\rho-1}{\rho}g(x)} (g(x)-g(a))^{\beta-1}\Big)(t)\\
&=D^{n,\rho,g}\frac{\Gamma(\beta)}{\Gamma(\beta+n-\alpha)\rho^{n-\alpha}}e^{\frac{\rho-1}{\rho}g(t)}(g(t)-g(a))^{n-\alpha+\beta-1} \\
&=\frac{\rho^n\Gamma(\beta)(n-\alpha+\beta-1)(n-\alpha+\beta-1)\cdot \cdot \cdot
   (\beta-\alpha)}{\rho^{n-\alpha}\Gamma(n-\alpha+\beta)}
\times e^{\frac{\rho-1}{\rho}g(t)}(g(t)-g(a))^{\beta-1-\alpha}\\
& =\frac{\rho^\alpha\Gamma(\beta)}{\Gamma(\beta-\alpha)}e^{\frac{\rho-1}{\rho}g(t)}(g(t)-g(a))^{\beta-1-\alpha}.
\end{align*}
Here, we have used the fact that $\displaystyle D^{\rho,g} \Big(h(t)  e^{\frac{\rho-1}{\rho}g(t)} \Big)=\rho \frac{h'(t)}{g^\prime (t)}  e^{\frac{\rho-1}{\rho}g(t)} $. 
\end{proof}

Below we present the semi--group property for  the general fractional proportional integrals of  a function with respect to another function.

\begin{theorem} \label{THM1} 
Let $\rho\in (0,1],~Re(\alpha)>0$ and $Re(\beta)>0$. Then, if $f$ is continuous and defined for $t \geq a$ or $t\le b$, we have
\begin{equation} \label{Left Semi integrals}
  ~_aI^{\alpha,\rho,g} (_{a}I^{\beta,\rho,g} f)(t)= ~_aI^{\beta,\rho,g} (_{a}I^{\alpha,\rho} f)(t)=(~_{a}I^{\alpha+\beta,\rho,g} f)(t)
\end{equation}
\noindent and 
\begin{equation}\label{Right Semi integrals}
  I_b^{\alpha,\rho,g} (I_b^{\beta,\rho,g} f)(t)= ~I_b^{\beta,\rho,g} (I_b^{\alpha,\rho} f)(t)=(I_b^{\alpha+\beta,\rho,g} f)(t).
\end{equation}
\end{theorem}
\begin{proof} We will prove \eqref{Left Semi integrals}. \eqref{Right Semi integrals} is proved similarly. Using the definition, interchanging the order and making the change of variable $z=\frac{g(u)-g(\tau)}{g(t)-g(\tau)}$, we get

\begin{align*}
 &~_aI^{\alpha,\rho,g} (_{a}I^{\beta,\rho,g} f)(t)\\
  &= \frac{1}{\rho^{\alpha+\beta}\Gamma(\alpha)\Gamma(\beta)}
  \int_a^t \int_a^ue^{\frac{\rho-1}{\rho}(g(t)-g(u))}e^{\frac{\rho-1}{\rho}(g(u)-g(\tau))}(g(t)-g(u))^{\alpha-1}\\
   & \times(g(u)-g(\tau))^{\beta-1}f(\tau)g'(\tau)d\tau g'(u)du\\
   &=\frac{1}{\rho^{\alpha+\beta}\Gamma(\alpha)\Gamma(\beta)}\int_a^t e^{\frac{\rho-1}{\rho}(g(t)-g(\tau))} f(\tau) \int_\tau^t
  (g(t)-g(u))^{\alpha-1} (g(u)-g(\tau))^{\beta-1}\\
  &\times g'(u)dug'(\tau) d\tau \\
   &= \frac{1}{\rho^{\alpha+\beta}\Gamma(\alpha)\Gamma(\beta)}
   \int_a^t e^{\frac{\rho-1}{\rho}(g(t)-g(\tau))}(g(t)-g(\tau))^{\alpha+\beta-1} f(\tau)g'(u)d\tau \\
   &\times \int_0^1 (1-z)^{\alpha-1} z^{\beta-1} dz\\
    &=\frac{1}{\rho^{\alpha+\beta}\Gamma(\alpha+\beta)}\int_a^t e^{\frac{\rho-1}{\rho}(g(t)-g(\tau))}(g(t)-g(\tau))^{\alpha+\beta-1} f(\tau)g'(\tau)d\tau\\
    &=(_{a}I^{\alpha+\beta,\rho} f)(t).
\end{align*}
\end{proof}

\begin{theorem}\label{THM2}Let $0\leq  m< [Re(\alpha)]+1$. Then, we have
\begin{equation}\label{D on L}
  D^{m,\rho,g} (_{a}I^{\alpha,\rho,g}f)(t)=(_{a}I^{\alpha-m,\rho,g}f)(t)
\end{equation}
and 
\begin{equation}\label{LD on L}
  ~_{\ominus}D^{m,\rho,g} (I_b^{\alpha,\rho,g}f)(t)=(I_b^{\alpha-m,\rho,g}f)(t)
\end{equation}
\end{theorem}

\begin{proof}
 Here we  prove \eqref{D on L}, while one can prove \eqref{LD on L} likewise. Using the fact that  $D_t^{\rho,g}e^{\frac{\rho-1}{\rho}(g(t)-g(\tau))} =0$), we have
\begin{align*}
 & D^{m,\rho,g} (_{a}I^{\alpha,\rho,g}f)(t) D^{m-1,\rho,g} (D^{\rho,g}~_{a}I^{\alpha,\rho,g}f)(t) \\
   &= D^{m-1,\rho,g} \frac{1}{\rho^{\alpha-1}\Gamma(\alpha-1)}\int_a^t e^{\frac{\rho-1}{\rho}(g(t)-g(\tau))} (g(t)-g(\tau))^{\alpha-2}f(\tau)g'(\tau)d \tau.
\end{align*}
Proceeding  $m-$times in the same manner  we obtain  (\ref{D on L}).
\end{proof}

\begin{corollary}\label{D on I}
Let $0<Re(\beta) < Re(\alpha)$ and $m-1<Re(\beta)\leq m$. Then, we have 
\begin{equation} \label{LD on LI}
_{a}D^{\beta, \rho,g} ~_{a}I^{\alpha,\rho,g} f(t)=~_{a}I^{\alpha-\beta,\rho,g} f(t)
\end{equation}
and 
\begin{equation} \label{RD on RI}
D_b^{\beta, \rho,g} I_b^{\alpha,\rho,g} f(t)=I_b^{\alpha-\beta,\rho,g} f(t).
\end{equation}
\end{corollary}
\begin{proof}
By the help of Theorem \ref{THM1} and Theorem \ref{THM2}, we have
\begin{eqnarray*}
  ~_{a}D^{\beta, \rho,g} ~_{a}I^{\alpha,\rho,g} f(t)&=& D^{m,\rho,g} _{a}I^{m-\beta,\rho,g} _{a}I^{\alpha,\rho,g}f(t)\\
   &=& D^{m,\rho,g}~_{a}I^{m-\beta+\alpha,\rho,g} f(t)=~_{a}I^{\alpha-\beta,\rho,g} f(t).
\end{eqnarray*}
This was the proof of \eqref{LD on LI}. One can prove \eqref{RD on RI} in a similar way.
\end{proof}

\begin{theorem} \label{THM4}
Let $f$ be integrable on $t\geq a$ or $t\le b$ and $Re[\alpha]>0, ~\rho \in (0,1],~~n=[Re(\alpha)]+1$. Then, we have 
\begin{equation}\label{LD on LI sameorder}
~_{a}D^{\alpha, \rho,g} ~_{a}I^{\alpha, \rho,g} f(t)=f(t)
\end{equation}
\noindent and 
\begin{equation}\label{RD on RI sameorder}
D_b^{\alpha, \rho,g} I_b^{\alpha, \rho,g} f(t)=f(t).
\end{equation}

\end{theorem}

\begin{proof}
By the definition and Theorem \ref{THM1}, we have
$$~_{a}D^{\alpha, \rho,g} ~_{a}I^{\alpha, \rho,g} f(t)=D^{n,\rho,g}~_{a}I^{n-\alpha, \rho,g}  ~_{a}I^{\alpha, \rho,g} f(t)= D^{n,\rho,g}~_{a}I^{n, \rho,g} f(t)=f(t).$$
\end{proof}

\section{The Caputo fractional proportional derivative of a function with respect to another function }
\begin{definition}
For $\rho \in (0,1]$ and $\alpha \in \mathbb{C}$ with $Re(\alpha)\geq 0$ we define the left derivative of  Caputo type as
\begin{align}\label{CFP}
 &(^{C}_{a}D^{\alpha,\rho,g} f)(t)=_{a}I^{n-\alpha,\rho,g} (D^{n,\rho,g}f)(t)\\\nonumber
 &=\frac{1}{\rho^{n-\alpha}\Gamma(n-\alpha)}\int_a^t e^{\frac{\rho-1}{\rho}(g(t)-g(s))}(g(t)-g(s))^{n-\alpha-1}(D^{n,\rho,g}f)(s)g'(s)ds.
\end{align}
Similarly, the right  derivative of  Caputo type ending  is defined by
\begin{align}\label{rCFP}
 &(^{C}D_b^{\alpha,\rho} f)(t)= I_b^{n-\alpha,\rho,g} (_{\ominus}D^{n,\rho,g}f)(t)\\\nonumber 
 &= \frac{1}{\rho^{n-\alpha}\Gamma(n-\alpha)}\int_t^b e^{\frac{\rho-1}{\rho}(g(s)-g(t))}(g(s)-g(t))^{n-\alpha-1}(~_{\ominus}D^{n,\rho,g}f)(s)g'(s)ds,
\end{align}
where $n=[Re(\alpha)]+1$.
\end{definition}

\begin{proposition} \label{4.2}
Let $\alpha, \beta \in \mathbb{C}$ be such that $Re(\alpha)> 0$ and $Re(\beta)>0$. Then, for any $\rho \in (0,1]$ and $n=[Re(\alpha)]+1$ we have
\begin{enumerate}

  \item $\big(^{C}_{a}D^{\alpha,\rho,g} e^{\frac{\rho-1}{\rho}g(x)} (g(x)-g(a))^{\beta-1}\big)(t)=\frac{\rho^\alpha\Gamma(\beta)}{\Gamma(\beta-\alpha)}e^{\frac{\rho-1}{\rho}g(t)}(g(t)-g(a))^{\beta-1-\alpha},$\\$~~~Re(\beta)> n.$
    \item $\big(^{C}D_b^{\alpha,\rho,g} e^{-\frac{\rho-1}{\rho}g(x)} (g(b)-g(x))^{\beta-1}\big)(t)=\frac{\rho^\alpha\Gamma(\beta)}{\Gamma(\beta-\alpha)}e^{-\frac{\rho-1}{\rho}g(t)}(g(b)-g(t))^{\beta-1-\alpha},$\\$~~~Re(\beta)> n.$
\end{enumerate}
For $k=0,1,\ldots,n-1$, we have $$\big(^{C}_{a}D^{\alpha,\rho,g} e^{\frac{\rho-1}{\rho}g(x)} (g(x)-g(a)^{k}\big)(t)=0\quad \mbox{and}\quad \big(^{C}D_b^{\alpha,\rho,g} e^{-\frac{\rho-1}{\rho}g(x)} (g(b)-g(x))^{k}\big)(t)=0.$$ In particular, $(~^{C}_{a}D^{\alpha,\rho} e^{\frac{\rho-1}{\rho}g(x})(t)=0$ and $(^{C}D_b^{\alpha,\rho} e^{-\frac{\rho-1}{\rho}g(x)})(t)=0$.
\end{proposition}
\begin{proof}
We only prove the first relation. The proof of the second relation  is similar. We have
\begin{align*}
&(^{C}_{a}D^{\alpha,\rho,g} e^{\frac{\rho-1}{\rho}g(x)} (g(x)-g(a))^{\beta-1})(t)= ~_{a}I^{n-\alpha,\rho,g} D^{n,\rho,g} \left[e^{\frac{\rho-1}{\rho}g(t)} (g(t)-g(a))^{\beta-1} \right]\\
   &=~_{a}I^{n-\alpha,\rho,g} 
   \left[ \rho^n (\beta-1)(\beta-2)\ldots(\beta-1-n) (g(t)-g(a))^{\beta-n-1} e^{\frac{\rho-1}{\rho}g(t)}\right] \\
   &= \frac{\rho^n (\beta-1)(\beta-2)\ldots(\beta-1-n)\Gamma(\beta-n)} {\Gamma(\beta-\alpha)\rho^{n-\alpha}} (g(t)-g(a))^{\beta-\alpha-1} e^{\frac{\rho-1}{\rho}g(t)}\\
   &= \frac{\rho^\alpha\Gamma(\beta)}{\Gamma(\beta-\alpha)}e^{\frac{\rho-1}{\rho}g(t)}(g(t)-g(a))^{\beta-1-\alpha}.
\end{align*}
\end{proof}

\section{Conclusions}
We have used the  proportional derivatives of a function with respect to another  to obtain left and right  generalized type of fractional integrals and derivatives involving two parameters $\alpha$ and $\rho$ and depending on a kernel function . The Riemann--Liouville and Caputo fractional derivatives in classical fractional calculus can  obtained as $\rho$ tends to $1$ and by choosing $g(x)=1$. The  integrals have the semi--group property and together with their corresponding derivatives have exponential functions as part of their kernels.  It should be noted that other properties of these new operators can be obtained by using  the Laplace transform proposed in \cite{fahd10}. Moreover, for a specific choice of $g$, the proportional fractional operators in the settings of Hadamard and Katugamplola can be extracted.
\vspace{2cm}

\noindent \textbf{Funding} This research was funded by the Deanship of Scientific Research at  Princess  Nourah  bint  Abdulrahman University through the Fast-track Research Funding Program.\\
\textbf{Availability of data and materials} Not applicable.\\
\textbf{Competing interests} The authors have no competing interest regarding this article.\\
\textbf{Authors contributions} All authors have done equal contribution in this article. All authors read and approved the last version of the manuscript.


\begin{thebibliography}{99}

\bibitem{podlubny} I. Podlubny,\textit{ Fractional Differential Equations} (Academic Press, San Diego CA, 1999).


\bibitem{Samko} S. G. Samko, A. A.  Kilbas, O. I. Marichev,\textit{ Fractional Integrals and Derivatives: Theory and Applications} (Gordon and Breach, Yverdon, 1993).



\bibitem{f1} R. Hilfer, \textit{Applications of Fractional Calculus in Physics} (Word Scientific, Singapore, 2000).

\bibitem{f222} L. Debnath, \emph{Recent applications of fractional calculus
to science and engineering}, Int. J. Math. Math. Sci.
 \textbf{(2003), Issue 54},  3413--3442.

\bibitem{f2} A. Kilbas, H. M. Srivastava, J. J. Trujillo, \textit{Theory and Application of Fractional Differential Equations} (North
Holland Mathematics Studies 204, 2006).





\bibitem{f3} R.L. Magin,\textit{ Fractional Calculus in Bioengineering} (Begell House Publishers, 2006).

\bibitem{had} A. A. Kilbas, \textit{Hadamard-type fractional calculus}, J. Korean Math. Soc. {\bf 38}(6) (2001), 1191--1204

\bibitem{Kat1} U. N. Katugampola, \textit{ New approach to generalized fractional integral}, Appl. Math. Comput. {\bf 218} (2011), 860--865.

\bibitem{Kat2} U. N. Katugampola,\textit{ A new approach to generalized fractional derivatives}, Bul. Math. Anal. Appl. {\bf 6 } (2014), 1--15.

\bibitem{fahd3} F. Jarad, T. Abdeljawad, D. Baleanu, \textit{On the generalized fractional derivatives and their Caputo modification}, J. Nonlinear Sci. Appl. \textbf{10 (5)} (2017), 2607-2619.

\bibitem{fahd1} F. Jarad, T. Abdeljawad,D. Baleanu, \textit{Caputo--type modification of the Hadamard fractional derivative}, Adv. Difference Equ. {\bf 2012}, 2012:142.
\bibitem{fahd11} F. Jarad, E. U\u{g}urlu, T. Abdeljawad,D. Baleanu, \textit{On a new class of fractional operators}, Adv. Difference Equ. {\bf 2018}, 2018:142.
\bibitem{fahd10} F Jarad, T Abdeljawad,  Generalized fractional derivatives and Laplace transform, Discrete and Continuous Dynamical Systems-S, doi:10.3934/dcdss.2020039.


\bibitem{kh} R. Khalil, M. Al Horani, A. Yousef,  and M. Sababheh, \textit{A new
Definition Of Fractional Derivative}, J. Comput. Appl. Math. {\bf 264} (2014), 65--70.

\bibitem{T11} T. Abdeljawad, \textit{On conformable fractional calculus }, J. Comput. Appl. Math. {\bf 279}  (2013), 57--66.





\bibitem{FCaputo} M. Caputo, M. Fabrizio, \textit{A new definition of fractional derivative without singular kernel}, Progr. Fract. Differ. Appl. {\bf 1} (2015), 73--85.

\bibitem{Losada} J. Losada, J. J. Nieto,\textit{ Properties of a new fractional derivative without singular kernel}, Progr. Fract. Differ. Appl. {\bf 1}(2015), 87--92.


\bibitem{TD ROMP} T. Abdeljawad, D. Baleanu, \textit{  On fractional derivatives with exponential kernel and their discrete versions}, Rep. Math. Phys. {\bf 80} 1 (2017), 11--27.

   \bibitem{Abdon} A. Atangana, D. Baleanu,\textit{ New fractional derivative with non-local and non--singular kernel}, Thermal Sci.  {\bf 20} (2016), 757--763.


 \bibitem{TD JNSA} T. Abdeljawad, D. Baleanu, \emph{Integration by parts and its applications of a new nonlocal fractional derivative with Mittag--Leffler nonsingular kernel}, J. Nonlinear Sci. Appl. \textbf{10 (3)} (2017), 1098--1107.

\bibitem{Anderson1} D. R. Anderson, D. J. Ulness, \emph{Newly defined conformable derivatives}, Adv. Dyn. Sys. App. \textbf{ 10 (2)} (2015) 109--137.

\bibitem{Anderson2} D. R. Anderson, \emph{Second--order self-adjoint differential equations using a proportional--derivative controller}, Comm. Appl. Nonlinear Anal. \textbf{24} (2017) 17--48.
 
\bibitem{fahd12} F. Jarad, T. Abdeljawad, J. Alzabut, \textit{Generalized fractional derivatives generated by a class of local proportional derivatives}, Eur. Phys. J. Special Topics {\bf 226} (2017), 3457-3471.

















  

































\end{thebibliography}
\end{document}